\documentclass[a4paper,10pt,twoside]{article}

\usepackage[body={175mm,250mm}]{geometry}
\usepackage{a1}
\usepackage[english]{babel}
\usepackage[latin1]{inputenc} %permite o uso de acentos
\usepackage{amsfonts,amssymb}
\usepackage{amsmath}
\usepackage{graphicx}	

\usepackage{makeidx}
\makeindex

\def\mapright#1#2#3{\smash{\mathop{\hbox to
#3{\rightarrowfill}}\limits^{#1}_{#2}}}

\def\mapleft#1#2#3{\smash{\mathop{\hbox to
#3{\leftarrowfill}}\limits^{#1}_{#2}}}

\def\mapright#1#2{\smash{\mathop{\hbox to 0.90cm{\rightarrowfill}}\limits^{#1}_{#2}}}
\def\mapleft#1#2{\smash{\mathop{\hbox to 0.90cm{\leftarrowfill}}\limits^{#1}_{#2}}}

\def\mapleftright#1#2{\smash{\mathop{\hbox to 0.80cm{\leftarrowfill \rightarrowfill}}\limits^{#1}_{#2}}}

\title{A short proof of the equivalence of any Reidemeister oriented move 3
\footnote{2010 Mathematics Subject Classification:
57M27 (primary), 57M25 (secondary)}}
\author{S\'ostenes Lins}
\date{\today}

\date{\today}

\begin{document}

\maketitle

\begin{abstract}
In this note we present a short proof that the 4 oriented Reidemeister 
moves of type 2 together with any one of the 8 oriented Reidemeister moves of type 3 are
sufficient to imply the other 7.
\end{abstract}

\section{Introduction}

In a recent paper, Polyak presented a minimal set of
Reidemeister moves sufficient to generate all the others of types 1,2,3, 
\cite{polyak2010minimal}. Moves of type 1 are special, in the sense that
looking for invariants of 3-manifolds we do not want to impose invariance
under move of type 1.
This is because the 3-manifold given by surgery in a
framed link do change under such a move \cite{kauffman1994tlr}.
Therefore, it is convenient
to have a minimal set of regular isotopy 
moves disregarding completely moves of type 1.
The author is aware that the result is well known. However, the proof 
that follows is
so short that it deserves to be mentioned.

\section{Proof of equivalence}
A Reidemeister move of type 3 has a triangular 
region over which is put upside down when the 
move is applied. There are 16 configurations of
triangular regions involved in oriented type 3 moves. The boundary
of the triangular region is formed by three curves which is ordered
in the following way: first the curve whose immediate extension goes
up and up, second the curve whose immediate extension 
goes once up and once down, third the curve whose immediate
extension goes down and down.
These curves can deformed to become line segments.
A triangular configuration present in a link diagram 
is locally rotated so that the up-up segment becomes
horizontal and the boundary forms a letter $\Delta$ (not a $\nabla$).
We encode such a triangular region by a sequence of 3 arrows. The first arrow
is horizontal and can be from west to east or from east to west;
the second arrow has its direction coinciding with the 
second line segment and the third arrow has its direction coinciding
with the third line segment. The 16 triangular configurations are
encoded and named as follows:
\noindent
\begin{center}
$
\begin{array}{|ccc|}\hline
a^\uparrow&\hspace{-0.5mm}:=\hspace{-0.5mm}
&\rightarrow \nearrow\nwarrow\\ \hline 
a^\downarrow&\hspace{-0.5mm}:=\hspace{-0.5mm}
&\leftarrow \swarrow\searrow\\ \hline
\end{array}
\begin{array}{|ccc|}\hline
b^\uparrow&\hspace{-0.5mm}:=\hspace{-0.5mm}
&\rightarrow \nwarrow\nearrow\\ \hline
b^\downarrow&\hspace{-0.5mm}:=\hspace{-0.5mm}
&\leftarrow \searrow\swarrow\\ \hline
\end{array}
\begin{array}{|ccc|}\hline
c^\uparrow&\hspace{-0.5mm}:=\hspace{-0.5mm}
&\rightarrow \swarrow\nwarrow\\ \hline
c^\downarrow&\hspace{-0.5mm}:=\hspace{-0.5mm}
&\leftarrow \nearrow\searrow\\ \hline
\end{array}
\begin{array}{|ccc|}\hline
d^\uparrow&\hspace{-0.5mm}:=\hspace{-0.5mm}
&\rightarrow \searrow\nearrow\\ \hline
d^\downarrow&\hspace{-0.5mm}:=\hspace{-0.5mm}
&\leftarrow \nwarrow\swarrow\\ \hline
\end{array}
\begin{array}{|ccc|}\hline
A^\uparrow&\hspace{-0.5mm}:=\hspace{-0.5mm}
&\rightarrow \searrow\swarrow\\ \hline
A^\downarrow&\hspace{-0.5mm}:=\hspace{-0.5mm}
&\leftarrow \nwarrow\nearrow\\ \hline
\end{array}
\begin{array}{|ccc|}\hline
B^\uparrow&\hspace{-0.5mm}:=\hspace{-0.5mm}
&\rightarrow \nearrow\searrow\\ \hline
B^\downarrow&\hspace{-0.5mm}:=\hspace{-0.5mm}
&\leftarrow \swarrow\nwarrow\\ \hline
\end{array}
\begin{array}{|ccc|}\hline
C^\uparrow&\hspace{-0.5mm}:=\hspace{-0.5mm}
&\rightarrow \nwarrow\swarrow\\ \hline
C^\downarrow&\hspace{-0.5mm}:=\hspace{-0.5mm}
&\leftarrow \searrow\nearrow\\ \hline
\end{array}
\begin{array}{|ccc|}\hline
D^\uparrow&\hspace{-0.5mm}:=\hspace{-0.5mm}
&\rightarrow \swarrow\searrow\\ \hline
D^\downarrow&\hspace{-0.5mm}:=\hspace{-0.5mm}
&\leftarrow \nearrow\nwarrow\\ \hline
\end{array}
$
\end{center}

A property of the 3-arrow notation is that it defines the eight
oriented forms of Reidemeister moves of type 3: a move $x$ is the passage
from $x^\uparrow$ to $x^\downarrow$ (or vice-versa), where $x \in \{a,b,c,d,A,B,C,D\}$.
The three arrows are reversed in the move.
Another property is that we get the mirror of a configuration
interchanging lower and upper case of the same symbol in
$\{a,c,A,C\}$.
Similarly for the symbols in $\{b,d,B,D\}$ but for these in addiction
we must also interchange $x^\uparrow$ and $x^\downarrow$, 
see  Fig. \ref{fig:thetas}. A {\em consistent digon} is one whose 2 
crossings have distinct signs. 
We say that a strand that crosses the digon is {\em consistent}
if the crossings are up-up or down-down. A consistent digon crossed by a consistent strand
is named a $\Theta$-configuration. 

\begin{proposition}
 There are sixteen $\Theta$-configurations.
\end{proposition}
\begin{proof}
By rotating we may suppose that the horizontal strand goes from east to west
in each $\Theta$-configuration. The horizontal strand can be $up$ or $down$.
The upper strand of the digon can be at the left or at the right. Finally,
the 4 directions of the two strands of the digon may occur independently. 
This yields $2 \times 2 \times 4=16$ $\Theta$-configurations.
\end{proof}

Suppose the two triangular regions of a $\Theta$-configuration
be labeled by $x^\uparrow$ (the upper triangle) and $y^\downarrow$
(the lower triangle). Then we say that
$(x^\uparrow,y^\downarrow)$ {\em are  $\Theta$-related}.
Fig. \ref{fig:thetas} shows that
if $(x^\uparrow,y^\downarrow)$ are  $\Theta$-related, then $(y^\uparrow,x^\downarrow)$ are  $\Theta$-related.

\begin{figure}[!h]
\begin{center}
\includegraphics[width=15cm]{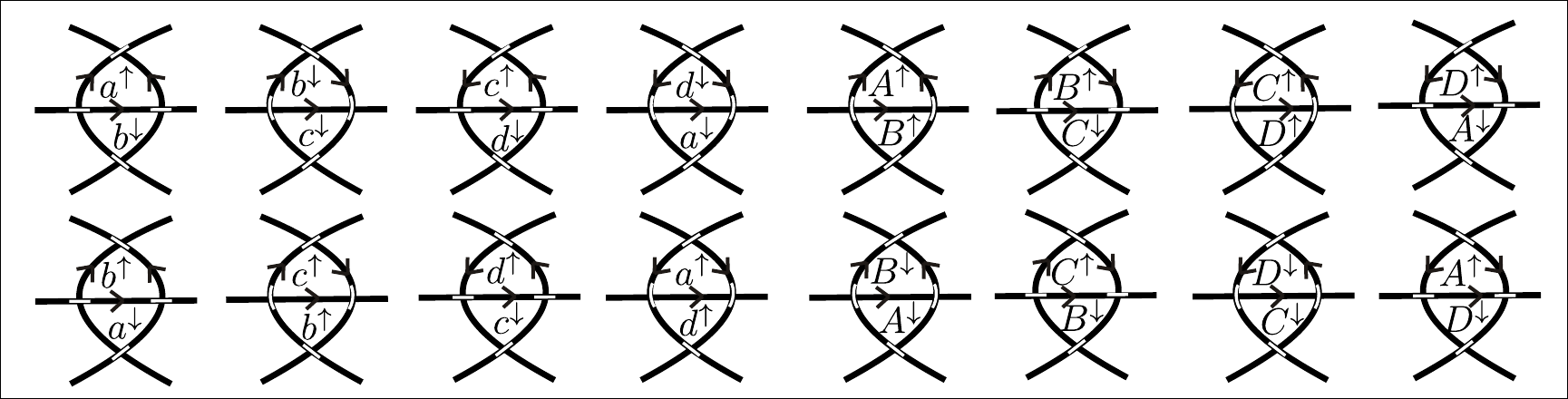} \\
\caption{\sf The sixteen $\Theta$-configurations}
\label{fig:thetas}
\end{center}
\end{figure}

For $x, y \in \{a,b,s,d,A,B,C,D\}$, we say that 
$y \Rightarrow_3 x$ if $x^\downarrow$ is obtained from $x^\uparrow$
by moves of type 2 and moves 
$y^\downarrow \to y^\uparrow$ and their inverses. 

\begin{lemma} If $(x^\uparrow,y^\downarrow)$ are  $\Theta$-related, then $y\Rightarrow_3 x$.
\label{lem:secondlemma}
\end{lemma}
\begin{proof}
The $x^\uparrow \to x^\downarrow$ move can be factored by a
complexifying type 2 move, the move $y^\downarrow \to y^\uparrow$ 
and a simplifying move of type2, see Fig. \ref{fig:yimpliesx}.
This establishes the result.
\end{proof}

\begin{figure}[!h]
\begin{center}
\includegraphics[width=9cm]{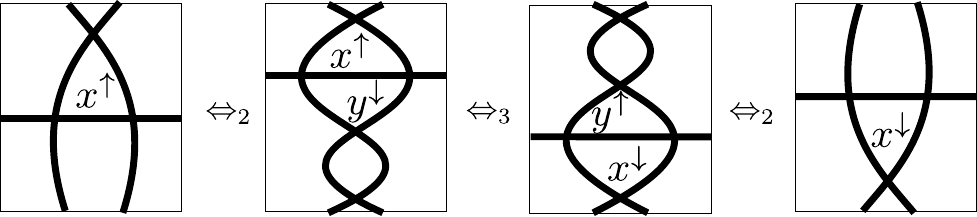} \\
\caption{\sf Proof that $y \Rightarrow_3 x$}
\label{fig:yimpliesx}
\end{center}
\end{figure}

\begin{corollary} If $(x^\uparrow,y^\downarrow)$ are  $\Theta$-related, 
then $y\Leftrightarrow_3 x$.
\end{corollary}
\begin{proof}
  By Lemma \ref{lem:secondlemma}, $y\Rightarrow_3 x$. 
If $(x^\uparrow,y^\downarrow)$ are  $\Theta$-related, then $(y^\uparrow,x^\downarrow)$ are  $\Theta$-related
by Fig. \ref{fig:thetas}.
As $(y^\uparrow,x^\downarrow)$ are  $\Theta$-related, by Lemma 
\ref{lem:secondlemma}, $x \Rightarrow_3 y$. In this way,
 $y\Leftrightarrow_3 x$.
\end{proof}

At this point, in face of the above corollary and of 
Fig. \ref{fig:thetas} we have
$a \Leftrightarrow b \Leftrightarrow c \Leftrightarrow d$ and 
$A \Leftrightarrow B \Leftrightarrow C \Leftrightarrow D$.

\begin{figure}[!h]
\begin{center}
\includegraphics[width=7cm]{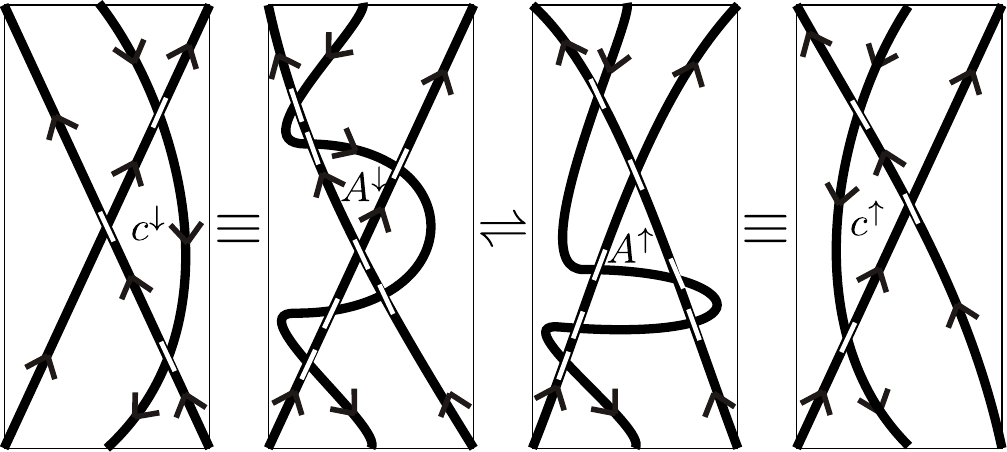} \\
\caption{\sf Proof that $A \Rightarrow c$ and of 
its mirrow, that $a \Rightarrow C$}
\label{fig:Aimpliesc}
\end{center}
\end{figure}

\begin{theorem}
 The set of 4 oriented Reidemeister moves of type 2 together with 
any one of the 8 oriented
 Reidemeister moves of type 3 implies the other 7.
\end{theorem}
\begin{proof}
 In Fig. \ref{fig:Aimpliesc} we prove that  $A \Rightarrow c$. 
Taking the mirrow of the moves involved
 we get that  $a \Rightarrow C$. Thus, the 8 
oriented Reidemeister moves of type 3 are equivalent given 
 that moves of type 2 are available. 
\end{proof}

\bibliographystyle{plain}
\bibliography{bibtexIndex.bib}

\end{document}